\documentclass{article}
\usepackage[utf8]{inputenc}

\usepackage[margin=1in]{geometry}  
\usepackage{graphicx}              
\usepackage{amsmath, comment}               
\usepackage{mathrsfs}
\usepackage{amsfonts}              
\usepackage{amsthm}                
\usepackage{amssymb}
\usepackage{amscd}
\usepackage{listings}
\usepackage{mathtools}
\usepackage{tikz}
\usetikzlibrary{arrows, matrix}
\usepackage{tikz-cd}

\newtheorem{thm}{Theorem}[section]

\newtheorem*{clm*}{Claim}

\newtheorem{defn}[thm]{Definition}
\newtheorem{rem}[thm]{Remark}
\newtheorem{fact}[thm]{Fact}
\newtheorem{leftfact}[thm]{\hspace*{-.88cm}Fact}
\newtheorem{ex}{Example}
\newtheorem{q}{Question}
\newtheorem{notation}{Notation}

\newcommand{\topstrut}{\rule{0pt}{2.7ex}}    

\newcommand\QQ{{\mathbb Q}}
\newcommand{\ZZ}{\mathbb{Z}}      

\newcommand\NN{{\mathbb N}}

\newcommand{\Z}{\mathbb{Z}}
\newcommand{\Q}{\mathbb{Q}}

\newcommand{\C}{\mathbb{C}}
\newcommand{\F}{\mathbb{F}}

\newcommand{\G}{\mathbf{G}}

\newcommand{\sha}{\rotatebox[origin=c]{180}{$\Pi\kern-0.361em\Pi$}}

\newcommand{\ms}{\mathsf}

\newcommand{\mc}{\mathcal}
\newcommand{\bb}{\mathbb}
\DeclareMathOperator{\Spf}{Spf}
\DeclareMathOperator{\Spm}{Spm}
\DeclareMathOperator{\Spa}{Spa}
\DeclareMathOperator{\Aut}{Aut}
\newcommand{\GL}{\mathrm{GL}}

\DeclareMathOperator{\Lie}{Lie}
\DeclareMathOperator{\fLie}{\mathscr{L}ie}
\DeclareMathOperator{\Spec}{Spec}
\newcommand{\llangle}{\langle\!\langle}
\newcommand{\rrangle}{\rangle\!\rangle}
\newcommand{\llbracket}{[\![}
\newcommand{\rrbracket}{]\!]}
\newcommand{\llbrace}{\{\!\{}
\newcommand{\rrbrace}{\}\!\}}

\usepackage{hyperref}
\hypersetup{
    colorlinks,
    citecolor=blue,
    filecolor=blue,
    linkcolor=blue,
    urlcolor=blue
}

\title{A Global Crystalline Period Map}
\author{Michael Neaton, Andreas Pieper, Catherine Ray}
\date{}

\usepackage{graphicx}

\begin{document}

\maketitle

\begin{abstract} The crystalline period map is a tool for linearizing $p$-divisible groups.  It has been applied to study the Langlands correspondences, and has possible applications to the homotopy groups of spheres. The original construction of the period map is inherently local.  We present an alternative construction, giving a map on the entire moduli stack of $p$-divisbile groups, up to isogeny, which specializes to the original local construction.
\end{abstract}

\tableofcontents

\section*{Acknowledgement} 
This paper is the cumulation of work from a Hopkins Project Group, which began at the 2019 Arizona Winter School on Topology and Arithmetic. We thank Jared Weinstein for suggesting possible appropriate sources for a global period map. We thank Eric Chen, Zixin Wendy Jiang, and Ningchuan Zhang for also being a part of our group. We further thank Eric Chen for being actively involved during the conference and writing the sections on both the Lubin-Tate tower and on the groups acting on it, and Ningchuan Zhang for teaching us the covariant definition of the crystalline period map and answering followup questions on it. The third author is supported by the NSF GRFP under Grant Number DGE 1842165.

\section{Introduction} 

We created a global version of the crystalline period map, otherwise called the Gross-Hopkins period map. Let's first review the setup of the classical period map for abelian varieties of dimension $g$ over $\C$. In this case, we have the following correspondence, the arrows of which are projection maps.

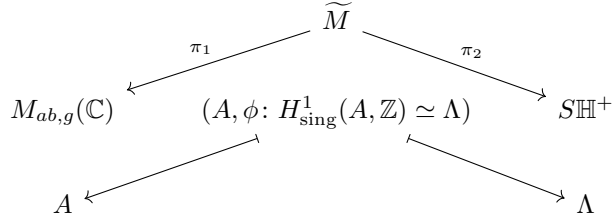
\begin{center}
\begin{tikzcd}
                        & \widetilde{M} \arrow[ld, "\pi_1"'] \arrow[rd, "\pi_2"]                                                                                   &               \\
{M_{ab, g}(\mathbb{C})} & {(A, \phi \colon H^1_{\mathrm{sing}}(A,  \mathbb{Z}) \simeq \Lambda)} \arrow[ld, maps to] \arrow[rd, maps to] & S\mathbb{H}^+ \\
A                       &                                                                                                                       & \Lambda      
\end{tikzcd}
\end{center}

Here, $M_{ab, g}$ is the moduli stack of abelian varieties of dimension $g$, $\widetilde{M}$ is the moduli space of abelian varieties of dimension $g$ with a specified presentation $\C^g/\Lambda$, and $S\mathbb{H}^+$ is Siegel upper half space. The period map associated to this correspondence is $\pi_2 \circ\pi_1^{-1}$. This map is defined up to an action of $\mathrm{Sp}_{2g}(\Z)$, which is changing the specification $\phi$ of symplectic basis for $H^1_{\mathrm{sing}}(A,  \mathbb{Z})$.

Similarly, we get a period map anytime we have a Hodge filtration. Here, ``Hodge filtration" is very loose: we don't require the exact sequence to split. In the case of complex abelian varieties above, the Hodge filtration does split; thus, for any such abelian variety $A$ we have

$$H^1_{\text{sing}}(A, \mathbb{Z}) \otimes \mathbb{C} \simeq (\Lie A)^* \oplus \Lie A^*.$$ 

\begin{rem} Note that $H^1(A, \mc{O}_A)/\Lambda \simeq Pic^0(A)$, and that $H_1(A, \Z) \hookrightarrow H^0(A, \Omega_A) = H^{1, 0}$. We get a Hodge filtration on the first cohomology.   \end{rem} 

We may attempt to do the same process for abelian varieties over $\mathbb{C}_p:=\widehat{\overline{\Q}}_p$.  We will transfer to considering the associated $p$-divisible groups, as this captures all of the $p$-torsion information.  Furthermore, we will restrict our attention to $p$-divisible groups of dimension $1$:  our reason for this restriction is that it evidently makes the problem easier, but it is also the principle concern for applications to homotopy theory.  Thus, let $G_0$ be a $p$-divisible group of height $h$ and dimension $1$.

We will now discuss many objects that will be defined later, in order to give the reader a roadmap of the aims and of the objects used in the paper. We may write down an analogous correspondence which captures the information of one deformation neighborhood of the $p$-torsion of a fixed abelian variety $A[p^\infty]$, or, more precisely, the deformation neighborhood of a 1-dimensional $p$-divisible group $G_0$. Let $G$ be a lift of $G_0$. The period map we will get from this correspondence is the crystalline period map.

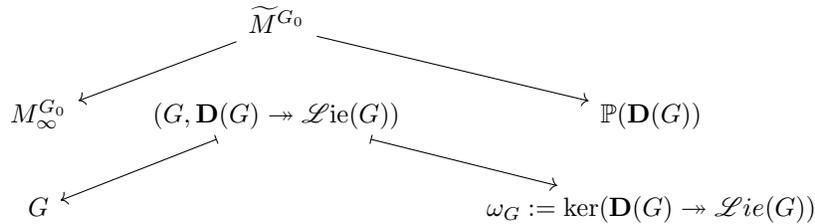
\begin{center}
\begin{tikzcd}
               & \widetilde{M}^{G_0} \arrow[ld] \arrow[rd]                                                   &                                                                 \\
M^{G_0}_\infty & {{(G, \mathbf{D}(G) \twoheadrightarrow \fLie(G))}} \arrow[ld, maps to] \arrow[rd, maps to] & \mathbb{P}(\mathbf{D}(G))                                              \\
G              &                                                                                             & \omega_G := \ker (\mathbf{D}(G) \twoheadrightarrow \mathscr{L}ie(G))
\end{tikzcd}
\end{center}

\begin{rem}
For the simplicity of the above diagram, when we write a lift $G$ above, we are implicitly including the quasi-isogeny and level structure, which we will fully exposit in Section \ref{sec: lttower}. 
\end{rem}

 \begin{center}
\begin{tabular}{ c|c } 
 \hline
 Classical Period Map & Crystalline Period Map \\
 \hline
$H^1_{\text{sing}}(A, \mathbb{Z})$ is locally constant & $D(G_0)$ is constant on $M^{G_0}_\infty$ \topstrut\\
Vary $\Lambda$  & Vary $\omega_G$ \\ 
 \hline
\end{tabular}
\end{center}

Prior to this paper, there was no known way to glue together the deformation neighborhoods of every $p$-divisible group $G_0$ into a coherent period map:  a mapping from the full moduli stack of all $p$-divisible groups to some lovely linear target. Our paper accomplishes this, using a Hecke stack as our moduli space of $p$-divisible groups with level structure (see Section~\ref{sec: glob}) and landing in a bundle of projective spaces.
    
The main interest in the crystalline period map has been its tantalizing connections to the local Langlands and Jacquet-Langlands correspondences, the unsolved Hecke Orbit conjecture \cite{oort}, and its so far disjoint application to the $K(h)$-local homotopy groups of spheres. 

The starring actor of the local Langlands and Jacquet-Langlands correspondences for $GL_h$ as approached by Harris-Taylor, Carayol, Drinfel'd, Beilinson, and others, is the analyitification of the Shimura variety $S_K$ (e.g., see page 69 of \cite{carayol}). This arises as the global counterpart to Lubin-Tate space $LT_h$ through the Cherednik-Drinfel'd theorem. The cohomology of $S_K$ is taken, and a global-to-local approach is applied to show that the local Langlands and Jacquet-Langlands correspondences hold through presenting $H^*_{\acute{e}t}(LT_h; \overline{\Q}_\ell)$ as a product representation. 

The $E_2$ page of the homotopy fixed point spectral sequence converging to the $K(h)$-local homotopy groups of spheres is the group cohomology of the Lubin-Tate action $H_{cts}^*(J, LT_*)$, i.e., the functions on Lubin-Tate space of a fixed $p$-divisible group $G_0$ being acted on by the automorphisms of $G_0$, where $J := \Aut_k(G_0)$. The crystalline period map linearizes the Lubin-Tate action by presenting it as the projectivization of $\mathbf{D}(G)$ being acted on by $J$. More specifically, it gives a $J$-equivariant Galois pro-\'{e}tale map from the generic fiber of the Lubin Tate tower $X$ to a projective space $Y$, where the actions of the Galois group (which is $\GL_h(\Q_p)$) and $J$ commute. Then, hopefully, the associated group cohomologies should satisfy $$H_{cts}^*(J, \mc{O}_Y) \simeq H_{cts}^*(J, \mc{O}_X)^{J \otimes \GL_h(\Q_p)}.$$ 
The main problems in applying this to the homotopic situation are two sides of the same coin: (1) the generic fiber of the Lubin-Tate tower has a very different collection of functions than the usual Lubin-Tate space, and (2) it's unclear what spectrum is associated to the generic fiber of the Lubin-Tate tower (in the way that Morava $E$-theory is associated to Lubin-Tate space). With the recent innovations in $p$-adic geometry, we may be able to understand how to go from the functions on the generic fiber of the Lubin-Tate tower to the functions on Lubin-Tate space via a clever resolution. This provides a glimmer of hope for the $E_2$ page. There is also hope for the $E_\infty$ page:  construct a spectrum associated to the generic fiber of the Lubin-Tate tower and relate its homotopy fixed points to those of $E$-theory. Both options have yet to be seriously explored. 

\section{The Crystalline Period Map}

Since it will be the source of our period map, we review the basic properties and definitions of the generic fiber of the Lubin-Tate tower.  Let $k$ be a finite field of characteristic $p$, and fix a $1$-dimensional formal group $G_0$ over $k$ of height $h$.

\subsection{The Lubin-Tate Tower}
\label{sec: lttower}

Let $\mathsf{Loc}_k$ be the category of Noetherian complete local $W(k)$-algebras $R$ with the data of an injection

$$ k \hookrightarrow R /\mathfrak{m}_R$$
where $\mathfrak{m}_R \subset R$ is the maximal ideal. For a ring $R$ in $\mathsf{Loc}_k$, we denote by $\kappa = R/\mathfrak{m}_R$ its residue field, and $\mathcal{O}=\mathcal{O}_R$ its ring of integers.

Consider the following stack classifying deformations of $G_0$ to formal groups over rings in $\mathsf{Loc}_k$ (with additional data). Such a functor is usually called a Lubin-Tate space and denoted $LT$.
$$M^{G_0}_{h,n}\colon R \longmapsto \left\{(G, \iota, \alpha) \middle|\, \begin{matrix} G \text{ is a formal group over }\mathcal{O} \text{ such that } G \times_R \kappa \cong G_0,\\ \iota\colon G \times_R \kappa_R \to G_0 \times_k \kappa \text{ is a quasi-isogeny,}\\
\alpha\colon (\ZZ/p^n\ZZ)^h \xrightarrow{\sim} G_0[p^n] \text{ is an isomorphism of group schemes.}
\end{matrix}\right\}$$

The functors $\{M^{G_0}_{h, n}\}_{n \in \NN}$ form an inverse system by
$$(G, \iota, \alpha) \mapsto (G, \iota, \alpha \circ p)$$
so we can assemble them into a tower
$$M_{h, \infty}^{G_0} := \varprojlim_n \, M^{G_0}_{n, h}$$
called the \textbf{Lubin-Tate tower}. Its moduli interpretation is deformations of our fixed formal group $G$ along with infinite level structure and a quasi-isogeny:
$$\mathcal{M}^{G_0}_{h,\infty}\colon R \longmapsto \left\{(G, \iota, \alpha) \middle|\, \begin{matrix} G \text{ is a formal group over }\mathcal{O} \text{ such that } G \times_R \kappa \cong G_0,\\ \iota\colon G \times_R \kappa_R \to G_0 \times_k \kappa \text{ is a quasi-isogeny,}\\
\alpha\colon \ZZ_p^h \xrightarrow{\sim} T_pG_0 \text{ is an isomosphism of group schemes}
\end{matrix}\right\}$$
where $TG_0:=\varprojlim_n G_0[p^n]$ denotes the Tate module of $G_0$. We take the generic fiber of this tower as our source. The moduli interpretation of the \emph{generic fiber} of the Lubin-Tate tower is deformations of our fixed formal group $G$ along with infinite \textit{rational} level structure and a quasi-isogeny:
$$\mathcal{M}^{G_0}_{h,\infty}\colon R \longmapsto \left\{(G, \iota, \alpha) \middle|\, \begin{matrix} G \text{ is a formal group over }\mathcal{O} \text{ such that } G \times_R \kappa \cong G_0,\\ \iota\colon G \times_R \kappa_R \to G_0 \times_k \kappa \text{ is a quasi-isogeny,}\\
\alpha\colon \Q_p^h \xrightarrow{\sim} T_pG_0 \otimes_{\Z_p} \Q_p \text{ is an isomosphism of }\Q_p\text{ vector spaces}
\end{matrix}\right\}.$$
Since we've taken the generic fiber, $p$ is now inverted, and the $p$-divisible groups are now being considered up to isogeny.

\begin{rem} We only use height as part of the moduli problem of the source of the period map in order to make the period map we will construct a Galois cover: it serves no other functional purpose. \end{rem}

\subsection{Construction of the Crystalline Period Map}

We summarize and add detail to the material of Dwork's original construction for height one \cite{katz}, as well as the Gross-Hopkins papers generalizing this construction to all heights \cite{gh}, \cite{gh2}.                                                                                                                                            
\subsubsection{Background on the Dieudonn\'{e} crystal}
\begin{rem} There are many ways to define the Dieudonn\'{e} crystal, all of which will work to do this construction. We will stick to the one which requires the least prerequisites. To see other ways of constructing the crystalline period map, contact C. Ray for a copy of {\normalfont Three perspectives on the crystalline period map}. \end{rem}

\begin{defn} Let $R$ be any ring. We define $\ms{VectGrpSch}_{/R}$ as the full subcategory of $\ms{GrpSch}_{/R}$ with objects given by finite direct sums of $\G_a$ (i.e., ``vector bundles"). \end{defn}

\begin{defn} Let $\ms{pDiv}_{/R}$ be the category of $p$-divisible groups over the ring $R$ up to isomorphism. For any $X \in \ms{pDiv}_{/R}$, we call the extension $$VX \to EX \to X$$ the \textbf{universal vector extension}  where $VG \in \ms{VectGrpSch}_{/R}$, if for any other extension $$V \to E \to X$$ of $X$ by $V \in \ms{VectGrpSch}_{/R}$, there are unique $f, F$ such that the diagram commutes.

\[ \begin{tikzcd}
VX \arrow[d, "f"] \arrow[r] & EX \arrow[d, "F"] \arrow[r] & X \arrow[d, "id"] \\
V \arrow[r]                          & E \arrow[r]                          & X                
\end{tikzcd} \]

\end{defn} 

This sequence always exists for formal groups. We linearize this sequence to: 

$$\fLie(VG) \to \fLie(EG) \to \fLie(G).$$

\begin{defn} The \textbf{Dieudonn\'{e} crystal} of a formal group $G$ over $R$ is $\fLie(EG)$. Note that it is an fppf sheaf once it is evaluated at any nilpotent thickening of $R$. \end{defn}

\begin{rem} We referred to the Dieudonn\'{e} crystal $\fLie(EG)$ as $\mathbf{D}(G)$ in the introduction. \end{rem}

\subsubsection{Construction of the Map}
Let us introduce a bit more notation before we construct the map. Take $G_{univ}$ to be the universal deformation of the formal group $G_0$.

\begin{notation} We define $$0 \to \mc{W} \to \mc{M} \to \mc{L} \to 0$$ to be the sequence of fppf sheaves:

$$0 \to \fLie(VG_{\mathrm{univ}}) \to \fLie(EG_{\mathrm{univ}}) \to \fLie(G_{\mathrm{univ}}) \to 0 \qquad \in \ms{Bun}_{/LT}$$
\textit{evaluated} on the thickening $A:= W(k)[[u_1, ..., u_{h-1}]] \twoheadrightarrow k[[u_1, ..., u_{h-1}]] =:A_0  $. (Before evaluating, its a sequence of crystals, not a sequence of fppf sheaves.) Note that $\Spf A \simeq LT$, where the latter is Lubin-Tate deformation space of a given fixed height $h$ formal group over $k$. We suppress $h$ in our notation, because $h$ is fixed.
\end{notation}

\begin{rem} The crystalline map is a ``Hodge variation map": that is, we wish to look at how $\mc{W}$ varies in $\mc{M}$. Further, we want this map to be $J$-equivariant, with respect to the action of $J := Aut_k(G_0)$ described in the next section.  \end{rem}

\begin{fact} Different lifts of $G_0$ up to isogeny correspond bijectively to inclusions $\fLie(VG) \hookrightarrow \fLie(EG)$. In other words, $\fLie(VG)$ is highly sensitive to the lift $G$ of $G_0$ chosen. \end{fact}

As we'll need it to define our period domain, let's briefly recall Grothendieck's projective space.

\begin{defn} Let $\mathcal{F}$ be a locally free sheaf over $X$. Then the (Grothendieck) projective space of $\mathcal{F}$, denoted by $\mathbb{P}(\mathcal{F})$, has the following moduli description.

$$\mathbb{P}\colon \mathcal{F} \longmapsto \left\{(\mathcal{L}, \pi) \, \middle| \, \begin{matrix} \mathcal{L} \text{ is a line bundle on } X\\ \pi\colon \mathcal{F} \to \mathcal{L} \text{ is surjective}\end{matrix}\right\}$$
\end{defn}

\begin{ex} If $X = \Spec k$, then we are describing lines over a point, and how they embed in a general space.  \end{ex}

Now we are ready to construct the map---let's go over the game plan. Recall that $J := \Aut_k(G_0)$. We will define the map on each rigidfied Lubin-Tate space $LT$, which will automatically define it on the generic fiber of the ridgidified Lubin-Tate tower. The rigidification is necessary for our map to exist, as we will see during Step 2 of the course of the construction.  To define our coveted $J$-equivariant map $$\pi \colon LT \to \mathbb{P}(M),$$ it is equivalent to take the following three steps: 
\begin{enumerate}
    \item  Define a $J$-equivariant bundle $\pi^*\mathcal{O}(1)$;
    \item  Show there exists a basis of sections which don't all vanish on any point of $\Gamma(\bb{P}(M), \mc{O}(1))$;
    \item  Show there is a surjective map from $\Gamma(\bb{P}(M), \mc{O}(1)) \to \Gamma(LT, L)$ respecting the $J$-action. 
\end{enumerate}

\begin{rem} Steps 2 and 3 are an unpacking of the term base-point free map.  \end{rem}

\begin{clm*} We can complete each of the above 3 steps and thus construct our map. \end{clm*}

\begin{proof} \noindent
\begin{itemize}
\item Proving 1: Take $\mc{L} := \fLie(G_{\mathrm{univ}})_{A \twoheadrightarrow A_0}$ as above. This map takes, for each point $a$ on Lubin-Tate space $LT$, the corresponding formal group $G_a$: $$G_a \mapsto \fLie(G_{univ})|_a = \fLie(VG_a) := \ker(\fLie(EG_a) \to \fLie(G_a)). $$ Geometrically, this map looks at how $\fLie(VG)$ is varying inside of $\fLie(EG)$ as we take different lifts of $G_0$. Further, $\mc{L}$ is automatically $J$-equivariant by construction.
\item Proving 2: Before we begin, we will need the following fact.
\begin{leftfact} Let $\mc{F}$ be locally free over an affine (formal) scheme $X$.  Then
$\Gamma(X, \mc{F}) \simeq \Gamma(\bb{P}(\mc{F}), \mc{O}_{\bb{P}(\mc{F})}(1))$. \end{leftfact}
Thus, if we define $V := \Gamma^{\nabla}(LT, \mc{M})$, then $\Gamma(\bb{P}(V), \mc{O}(1)) \simeq V$. 

Further, note that $$V:= \Gamma^{\nabla}(LT, \mc{M}) \hookrightarrow \Gamma(LT, \mc{M}) \simeq \Gamma(\bb{P}(M), \mc{O}(1)).$$ Thus, if we construct a basis of horizontal sections on $V$ which don't vanish at any point, we may push it forward to a basis of sections which don't vanish at any point on $\Gamma(\bb{P}(M), \mc{O}(1))$.

\begin{leftfact} $V$ doesn't have a basis of horizontal sections, but it does when we rigidify (see Remark~\ref{rigidrem}).\end{leftfact} 

We rigidify everything; now $V^{\mathrm{rig}}$ has a basis of sections, and since it includes into $\Gamma(LT^{\mathrm{rig}}, M^{\mathrm{rig}})$, this also has a basis of sections. 
\item Proving 3:  Let us take the long exact sequenced induced from the short exact sequence of sheaves $$\mc{L}^{\mathrm{rig}} \to \mc{M}^{\mathrm{rig}} \to \mc{L}^{\mathrm{rig}}.$$ To be precise, we consider
$$\Gamma(LT^{\mathrm{rig}}, \mc{W}^{\mathrm{rig}}) \to \Gamma(LT^{\mathrm{rig}}, \mc{M}^{\mathrm{rig}}) \to \Gamma(LT^{\mathrm{rig}}, \mc{L}^{\mathrm{rig}}) \to R^1\Gamma(LT^{\mathrm{rig}}, \mc{W}^{\mathrm{rig}}) \to \cdots.$$
Note that $R^1\Gamma$ of a quasicoherent sheaf over an affine space vanishes. Therefore, $\Gamma(LT^{\mathrm{rig}}, \mc{M}^{\mathrm{rig}}) \to \Gamma(LT^{\mathrm{rig}}, \mc{L}^{\mathrm{rig}})$ is surjective.
\end{itemize}
\end{proof}

\begin{rem} \label{rigidrem} Warning: It's a hard fact to justify rigidification as a necessity. This reqiures us to use a different definition of the Dieudonn\'{e} crystal. We use that $$ \fLie(EG_{univ})_{A \twoheadrightarrow A_0} \simeq PH^1_{dR}(G^t/A),$$ where $G^t$ is the Cartier dual.  We know that $PH^1_{dR}(G^t/A)$ has a basis of horizonal sections in the form of $p$-typical ``pre-logarithms" \begin{align*} 
\ell_0(x) &= \sum_k \frac{x^{(p^{h})^k}}{p^k}\\
\ell_i(x) &= \frac{\ell_0(x^{p^i})}{p} \qquad i = 1, ..., h-1 
\end{align*} which are power series. These power series do not converge on all of $LT$.  But, by a theorem of Katz (3.1.1 in \cite{katz}), these power series converge on all of $LT$ when we ``invert p" or ``take the generic fiber"; that is, we rigidify the formal scheme $\Spf W(k)\llbracket u_1, ..., u_{h-1}\rrbracket$ to the rigid analytic scheme $\Spm W(k)\llangle u_1, ..., u_{h-1}\rrangle$. Here, $W(k)\llangle u_1, ..., u_{h-1}\rrangle$ is the ring of convergent power series over the unit open polydisc. Note that this ring is denoted by $K\llbrace u_1,\cdots,u_{h-1}\rrbrace$ in \cite{katz}. \end{rem}

\subsection{Compatibility with Group Actions}

Now that we've defined all the maps in the triangle 
$$
\begin{tikzcd}
 & \mathrm{LT}_h \arrow{dr}{\pi_{\mathrm{GH}}} \\
\mathcal{M}^{G_0}_{h,\infty} \arrow{ur}{f} \arrow{rr}{\pi} && \mathbb{P}.
\end{tikzcd}
$$
We will study the groups that act on these spaces for which the maps involved are equivariant. In particular, we will see that all three maps are (pro-finite) torsors for some $p$-adic group. 

We will be interested in the following three groups: $J = \mathrm{Aut}_k(G_0)$ the automorphism group of $G_0$ over $k$, $\mathrm{GL}_h(\ZZ_p)$, and $W_{\QQ_p}$ the Weil group of $\QQ_p$. The first two groups act in an obvious way on the Lubin-Tate tower as follows: for a point $(G, \iota, \alpha)$ in the Lubin-Tate tower, an element $\sigma \in J$ and an element $g \in \mathrm{GL}_h(\QQ_p)$, the actions are
$$\sigma \cdot (G, \iota, \alpha) = (G, \sigma \circ \iota, \alpha),$$
$$g \cdot (G, \iota, \alpha) = (G, \iota, \alpha \circ g).$$

The action of the Weil group is more subtle. Recall that the Weil group fits inside a short exact sequence
$$0 \longrightarrow I_{\QQ_p} \longrightarrow W_{\QQ_p} \longrightarrow \langle \mathrm{Frob}\rangle^\ZZ \longrightarrow 0.$$
Let $w$ be an element in $W_{\QQ_p}$, and let $n(w)$ be its image in $\langle\mathrm{Frob}\rangle^\ZZ$. Then $w$ acts on a point in the Lubin-Tate tower by
$$w \cdot (G, \iota, \alpha) = (G^w, \iota^{w} \circ \mathrm{Frob}^n, \alpha^w).$$

\begin{rem} By understanding the $\overline{\Q}_{\ell}$ \'{e}tale cohomology of the Lubin-Tate tower as a product representation of the above three groups, one realizes both the local Langlands and Jacquet-Langlands correspondences for $GL_h$, for all $h \geq 0$. \end{rem}

\section{The Global Crystalline Period Map}

We will first recount for the reader the background needed to understand our construction of the global crystalline period map. We set the stage with two definitions of the Fargues-Fontaine curve and their relation, in order to establish how $p$-divisible groups are bundles (with modification) on the Fargues-Fontaine curve. With this viewpoint established, we give a brief idea of what we mean by a ``global crystalline period map," and then give the reader some background on the source of our period map, a Hecke stack which we call $\mathrm{Hecke}$. Finally, we describe our construction.

\subsection{The Fargues-Fontaine Curve $X_{C}$}

\subsubsection{Adic Perspective on $\mc{X}_{C}$}

The illustrative picture we have in mind in this section is Figure 5 page 84 of \cite{adicpic} (this figure also appears in various other guises elsewhere in \cite{adicpic}). Let $p$ be a fixed prime. Let $C$ be a complete algebraically closed field which contains $\Q_p$, e.g., $\widehat{\overline{\Q}}_p$. We take its tilt $C^\flat$ to get a complete algebraically closed field which contains $\F_p((T))$.  Let $W(\mathcal{O}_{C^\flat})$ be the Witt vectors of the ring of integers of $C^\flat$. 

Let $\varpi$ be a pseudo-uniformizer of $\mc{O}_{C^\flat}$ (recall a pseudo-uniformizer is simply any old element with norm in $(0,1)$, the choice of which doesn't matter). Let $x_k \in \Spa(W(\mc{O}_{C^\flat})$ be the unique non-analytic point, which is the valuation which factors through the residue field (i.e., where $p=\varpi=0$).  Any given valuation $x$ has a maximal generalization $\widetilde{x}$, which is necessarily of rank one.

\begin{rem}  We will use the standard notation in adic theory, where if $x\colon R\to\Gamma$ is a valuation, we write $|f(x)|$ for the value in $\Gamma$ of $x$ evaluated on $f\in R$.  The notation comes from thinking of the ring $R$ as a ring of functions, and thinking of $x$ as an absolute value.  This tends to improve intuition.
\end{rem}

Now we are ready to define the map $\kappa$, which sends a valuation to its absolute value.

\begin{align*} 
\kappa\colon \Spa(W(\mc{O}_{C^\flat}))\setminus\{x_k\} & \to [0, \infty] \\
x & \mapsto \frac{\log|\varpi(\widetilde{x})|}{\log|p(\widetilde{x})|}
\end{align*}

\begin{rem}  Looking at Figure 5 in \cite{adicpic}, one can see $\kappa$ as the ``angle''.  Then, $x_k$ is the ``origin'' with $p=\varpi=0$:  thus, we remove this point as it doesn't have a well-defined ``angle''.
\end{rem}

We now take $\mc{Y} \colon= \kappa^{-1}((0, \infty))$, which in effect removes the axis lines $\varpi = 0$ and $p=0$. In other words, $\mc{Y} \colon= \Spa(W(\mc{O}_{C^{\flat}}))\setminus \{p\varpi=0\}.$ 

Note if we let Frobenius act on $[0, \infty]$ via multiplication by $p$, then $\kappa$ is Frobenius equivariant. Thus, $\mc{Y}$ carries an action of Frobenius $\phi$. Note $\mc{Y}$ can be pictured as a minimal hypersurface on a coil, where the action of Frobenius shifts down. In other words, Frobenius acts by a full $360^\circ{}$ rotation. The quotient of $\mc{Y}$ by integer powers of Frobenius resembles an open disk. This ``open disk" is the Fargues-Fontaine curve. 

\begin{defn} The \textbf{Fargues-Fontaine curve} is defined as $\mc{X}_{R} := \mc{Y}/\phi^\Z$ \end{defn}

\begin{rem} $\mc{X}_{C}$ is called a curve because it is regular, Noetherian, dimension 1, and has a ``degree" function. However, it is not of finite type. \end{rem}

\begin{rem} Note that the relative Fargues-Fontaine curve is obtained the same way as the Fontaine-Fargue curve, we just begin with $\Spa(W(\mc{O}_{C^\flat}), W(\mc{O}_{C^\flat})^+)$ instead. \end{rem}

\subsubsection{Scheme Perspective on $X_{C}$}

There is another ``equivalent" way to define the Fargues-Fontaine curve, which we briefly exposit here while maintaining the adic perspective. We denote the adic construction by $\mc{X}_C$, and the scheme-theoretic construction by $X_C$. Otherwise, we continue to use the notation defined at the beginning of the previous section. We begin with $W(\mc{O}_{C^\flat})$, commonly called $A_{\mathrm{inf}}$. Let $\theta\colon W(\mathcal{O}_{C^\flat})\twoheadrightarrow\mathcal{O}_C$ be the standard map from $p$-adic Hodge theory.  We will only really use the fact that $\theta$ is surjective, and that its kernel is $\varpi - p$.  The corresponding valuation coming from $\theta$, $x_C$, satisfies $\kappa(x_C)=1$.

We invert $\ker \theta$, and get $B_{\mathrm{cris}}^{+} := W(\mc{O}_{C^{\flat}})[(\ker \theta)^{-1}]$, this is equivalent to removing the axis $\varpi = 0$. Then, we invert $p$ to remove the axis $p=0$, $B_{\mathrm{cris}} := W(\mc{O}_{C^{\flat}})[(\ker\theta)^{-1}][p^{-1}]$. By removing both lines, we have also removed the unique non-analytic point $x_k$ which is at $\varpi=p=0$. Thus, $\Spec B_{\mathrm{cris}}$ is the scheme theoretic equivalent of $\mc{Y}$ in the above section. Restricting to where Frobenius acts like powers of $p$ on $B_{\mathrm{cris}}$ is equivalent to the adic approach of taking the quotient of $\mc{Y}$ by Frobenius. 

\begin{defn} The \textbf{Fargues-Fontaine curve} is 
$$X_{C} := \mathrm{Proj}\ \bigoplus_{d \geq 0} \left( W(\mc{O}_{C^\flat})\left[\frac{1}{\ker \theta}\right]\left[\frac1p\right] \right)^{\phi = p^d}. $$ \end{defn}

\subsubsection{$p$-divisible Groups as Bundles on $X_{C}$ with Modifications} 

We have these two different Fargues-Fontaine curves, but we want to consider bundles on them.  The adic perspective is better, because it connects better to the perfectoid picture.  Luckily, the categories of vector bundles over these objects are equivalent by a GAGA result.

Let $G \in \mathsf{pDiv}_{/\mc{O}_C}$, and $G_0$ be its corresponding special fiber in $\mathsf{pDiv}_{/ \mc{O}_{C^\flat}}$. 

\begin{rem} In the language from the beginning of this paper, for $k$ a characteristic $p$ field, $G_0 \in \ms{pDiv}_{/k}$, and $G$ is a characteristic 0 lift of $G_0$ (for example, to $W(k)$). \end{rem}

For those encountering a $p$-divisible group presented as a bundle $\mc{E}(G)$ with extra data for the first time, we give a conceptual overview before introducing its definition. We freely move between the adic and scheme theoretic perspectives. 

\begin{enumerate}
\item We begin with the crystal $\fLie(EG_0)$ evaluated on a chosen pd-thickening of $\mc{O}_{C^\flat}$.  This is called the Dieudonn\'{e} module $\mathbf{D}(G)$ (it doesn't depend on choice of lift either). This is a $W(\mc{O}_{C^\flat})$-module, and thus a quasi-coherent sheaf $D$ on $\Spa(W(\mc{O}_{C^\flat}))$, which carries  an action of Frobenius. 
\item We pullback this sheaf $D$ (before modding out by Frobenius) via $$ \mc{Y}  \xrightarrow{f} \Spa(W(\mc{O}_{C^\flat})) $$ The Frobenius comes with the pullback. This is our vector bundle $\mc{E}(G) := f^*D$.
\item $\mc{E}(G)$ only remembers the special fiber of the $p$-divisible group. The modification remembers the choice of the characteristic 0 lift $G$. Note everything is up to isogeny, as we inverted $p$.
\end{enumerate} 

Equipped with intuition, we return to rigor. We define the vector bundle $\mathcal{E}(G)$ associated to a $p$-divisible group $G$ as follows. 

\begin{defn} The vector bundle associated to $G \in \mathsf{pDiv}_{/\mc{O}_C}$ is $$\mathcal{E}(G) := \left(\bigoplus_{d \geq 0} \mathbf{D}(G)\left[\frac1p\right]^{\phi = p^d}\right)\hspace*{-.07cm}\raisebox{2.2ex}{$\widetilde{\hspace*{.3cm}}$}$$ where $\widetilde{M}={M}^{\widetilde{\hspace*{.15cm}}}$ denotes the sheaf associated to the module $M$.\end{defn}

Note that $\mc{E}(G)$ is not equivalent to the datum of a $p$-divisible group. To be equivalent, we must also equip the bundle $\mathcal{E}(G)$ with a ``modification''. In this case, a modification is a map from the Tate module of the $p$-divisible group $G$ to the pullback of our bundle. 

Let us now set up the tools we need to define the modification. Our exposition of modifications is mostly an artistic interpretation of Section 5.1 of \cite{moduli}. Note that the Dieudonn\'{e} module up to isogeny $\mathbf{D}(G)$ over a complete algebraically closed field splits into $\mathbf{D}(G) = \bigoplus_{r, s} \mathbf{D}(G)^{\phi^r = p^s}$. We may map from $\mathbf{D}(G) \to \mathbf{D}(G)^{\phi = p^d}$ by projecting onto only the slope $d$ components. As before, we let $TG$ be the Tate module of the $p$-divisible group $G$. We may then construct the following composite map for each $d$: 

$$a^d\colon TG \to G \to \mathbf{D}(G) \twoheadrightarrow \mathbf{D}(G)^{\phi = p^d}.$$

Let $\mathcal{F}(G) := TG  \otimes_{\mc{O}_C} \mc{O}_{X_{C}}$. Then putting together the $a^d$ induces a map of quasi-coherent sheaves on $X_{C}$:

$$b\colon \mathcal{F}(G) \hookrightarrow \mathcal{E}(G).$$
Let $i_\infty\colon C \hookrightarrow X_{C}$ be the inclusion of the point at infinity into the relative Fargues-Fontaine curve. We can think of this in two ways: (1) it maps the special valuation $\theta$ composed with $\mc{O}_C \hookrightarrow C$ to the point at infinity; (2) this is the inclusion of at least a point of the $\varpi = p$ line into the Fargues-Fontaine curve.

When we take the pullback of $b$ over $i_\infty$, we get 

$$i_\infty^*b\colon TG \hookrightarrow i_\infty^*\mathcal{E}(G).$$ 

We will use the maps $a^d, b$ and $i_\infty^*b$ during our definition of the Hecke stack in Section \ref{sec: hecke}.

\begin{defn} Let $\mc{E}$ and $\mc{F}$ be vector bundles over a stack $X$. We say that $\mc{E}$ is a modification of $\mc{F}$ along the points $x, y \in X$ if there is an isomorphism: $$f\colon \mc{F}|_{X\setminus\{x\}} \simeq \mc{E}|_{X \setminus\{y\}}.$$  \end{defn}

\subsection{The Meaning of Globalizing a Period Map}
\label{sec: glob}

We discussed in Section \ref{sec: lttower} the deformation space $\mathcal{M}^{G_0}_{h, \infty}$ of a $p$-divisible group $G_0$ with a fixed height $h$ and infinite level structure.  In this section, we will consider instead the entire moduli stack of $p$-divisble groups of height $h$ with infinite level structure,  $\mathcal{M}_{h, \infty}$. We seek a globalization of the crystalline period map. Heuristically, what we mean by globalization of the period map is: 
\[
\begin{tikzcd}
{\mathcal{M}^{G_0}_{h, \infty}} \arrow[r, dotted] \arrow[d, "\pi_{\mathrm{GH}}"]  & {``\mathcal{M}_{h, \infty}"} \arrow[d, "\pi_{\mathrm{GH}}^{\text{glob}}", dotted] \\
\mathbb{P}(\mathbf{D}(G_0)) \arrow[r, dotted]                               & {\text{some sort of bundle of } \mathbb{P}^{h-1}}.                                   
\end{tikzcd}
\]

Since the stack $\mathcal{M}_{h, \infty}$ is \textit{highly} infinite as a spectral space, we instead look for an appropriate replacement for the source of the map.  We pause to consider what such a thing might look like. Its $\F_q$ points would be \textit{isogeny} classes of Dieudonn\'{e} modules of rank $h$. 

\begin{defn} The \textbf{Kottwitz set} $B(\GL_h)$ is defined as the set of $\sigma$-conjugacy classes of $\GL_h(W(\overline{\F}_p))$, where $\sigma$ is the Frobenius over $\overline{\F}_p$. This classifies isocrystals by their slope decomposition. \end{defn}

\begin{defn} 
If we only wish to look at isocrystals corresponding to Dieudonn\'{e} modules, we restrict our slopes to lay between $[0,1]$. Let $B'(GL_h)$ be the set of $\sigma$-conjugacy classes of $\GL_h(W(\overline{\F}_p))$ with slopes lying in $[0,1]$.
\end{defn}

\begin{defn} 
$\mathrm{Bun}_{h}(X_{W(\overline{\F}_p)})$ is an Artin stack in diamonds, whose underlying set of points is $B'(\GL_h)$. $\mathrm{Bun}_{h}(X_{W(\overline{\F}_p)})$ classifies vector bundles of rank $h$ on the Fargues-Fontaine curve. 
\end{defn} 

This is the subject of Fargues' work on the geometrization of the Langlands program over $\Q_p$. As we discussed earlier, $p$-divisible groups appear as modifications of vector bundles.

There is a Hecke stack classifying such modifications, which surjects onto the stack $\mathrm{Bun}_h(X_{W(\overline{\F}_p)})$. This Hecke stack satisfies the requirement that its $\F_q$ points are isomorphism classes of Dieudonn\'{e} modules of rank $h$. 

\subsubsection{A Cursory Introduction to our Hecke Stack}
\label{sec: hecke}
Let's briefly discuss our Hecke stack for the algebraic group $\mathrm{GL}_h$. We will use it in the least generality, and thus only define what we use.

\begin{defn} We define $\mathrm{Hecke}$ by its functor of points as follows. Let $C \in \mathrm{Perf}$, and $x, y \in X_{C}$. We take $\mc{E}, \mc{F} \in \mathrm{Bun}_h(X_{C})$.

$$\mathrm{Hecke}(C) := \{ (\mc{E}, \mc{F}, f \colon \mc{F}|_{X_{C} \setminus\{x\}} \simeq \mc{E}|_{X_{C} \setminus\{y\}}) \}.$$

\end{defn}

We then get a Hecke correspondence:

\begin{center} 
\begin{tikzcd}
                                 & \mathrm{Hecke} \arrow[ld] \arrow[rd]                             &                                  \\
\mathrm{Bun}_h(X_{C}) & {\{\mc{E}, \mc{F}, f \}} \arrow[ld, maps to] \arrow[rd, maps to] & \mathrm{Bun}_h(X_{C}) \\
\mc{E}                           &                                                                  & \mc{F}.                          
\end{tikzcd}
\end{center}

\subsubsection{The Construction of the Global Crystalline Period Map}

Recall that in the case of the (local) crystalline period map, the universal vector extension of $G_{\mathrm{univ}}$ gives us the short exact sequence 
$$\fLie(VG_{\mathrm{univ}}) \hookrightarrow \fLie(EG_{\mathrm{univ}}) \twoheadrightarrow \fLie(G_{\mathrm{univ}}),$$ 
evaluated on the pd-thickening $ W(k)\llbracket u_1, ..., u_{h-1}\rrbracket \twoheadrightarrow k\llbracket u_1, ..., u_{h-1}\rrbracket$. This gives us a period map to $\mathbb{P}(\fLie(EG_{\mathrm{univ}}))$, via $$G_a \mapsto \ker(\fLie(EG_a) \twoheadrightarrow \fLie(G_a)).$$ 

In the global setting, we try the following, which we phrase in terms of $C$-points for a perfectoid space $C$. Let $G$ be a $p$-divisible group over $C$, and $\fLie(G)$ its Lie algebra. Let $\mathcal{E}(G)$ be the corresponding vector bundle on $X_{C}$.  Then we have a map of quasi-coherent sheaves on $C$:

\[
\begin{tikzcd}
i_\infty^*\mathcal{E}(G) \arrow[r] \ar["\simeq \text{by S.W.}"]{d} & i_\infty^*(i_\infty)_*\fLie(G)  \arrow[d, "\simeq"]\\
\fLie(EG) \otimes C                 &                                      \fLie(G)     
\end{tikzcd}
\]
where the first map is $i_\infty^*$ applied to the map
$$\mathcal{E}(G) \longrightarrow (i_\infty)_*\fLie(G)$$ 
of sheaves on $X_{C}$, and the second map comes from the adjunction $(i_\infty^*, (i_\infty)_*)$.

\begin{defn} Let $\mathrm{Perf}$ be the category of perfectoid spaces. A \textbf{diamond} is a pro-\'etale sheaf on $\mathrm{Perf}$ such that it may be represented by the quotient of a perfectoid space by a pro-etale equivalence relation. \end{defn}

The vector bundle associated to a $p$-divisible group is naturally a diamond. We take the fully faithful embedding of the local crystalline period map into the category of diamonds. 

\begin{defn} We define the diamond $\mathrm{PerDom}$, our global period domain, by its functor of points as follows: for $C \in \mathrm{Perf}$,

$$\mathrm{PerDom}(C) := \coprod_{\mathcal{H} \in \mathrm{Bun}_h(X_{C})} \mathbb{P}(i_\infty^* \mathcal{H}).$$ 

We may equivalently define $$\mathrm{PerDom} := \mathbb{P}(i_\infty^* \mathrm{EBun})$$ where $\mathrm{EBun}$ is the universal vector bundle on $\mathrm{Bun}_h(X_{C})$. 

\end{defn}

\begin{thm} \label{diagram} By construction, the following is a commutative diagram of diamonds. 

\[
\begin{tikzcd}
{\mathcal{M}^{G_0, {\diamond}}_{h, \infty}} \arrow[rr] \arrow[d, "\pi_{\mathrm{GH}}^{\diamond}"] &  & \mathrm{Hecke} \arrow[d, "\pi_{\mathrm{GH}}^{\mathrm{glob}}"] \\
\mathbb{P}^{h-1 \text{ } \diamond} \arrow[rr]                                &  & \mathrm{PerDom} \arrow[d, "\mathrm{proj}"]           \\
                                                     &  & \mathrm{Bun}_h(X_{C})                            
\end{tikzcd}
\]

The maps are as follows:

\[
\begin{tikzcd}
                                                                 & G \arrow[rrdd, maps to] &                                                                         &                                                     \\
{(G, \iota, \alpha)} \arrow[ru, maps to]                         & \mathcal{M}_h \arrow[rd]          &                                                                         &                                                     \\
{\mathcal{M}^{G_0}_{h, \infty}}^{\diamond} \arrow[rr] \arrow[d, "\pi_{\mathrm{GH}}^{\diamond}"'] \arrow[ru] &                         & \mathrm{Hecke} \arrow[d, "\pi_{\mathrm{GH}}^{\mathrm{glob}}"'] & ( \mc{E}(G), \mc{F}(G), TG \hookrightarrow i_\infty^*\mathcal{E}(G)) \arrow[d, maps to] \\
\mathbb{P}^{h-1 \text{ } \diamond} \arrow[rr]                                            &                         & \mathrm{PerDom} \arrow[d, "\mathrm{proj}"']                                 & \mathbb{P}(i_\infty^* \mathcal{E}(G)) \arrow[d]                 \\
                                                                 &        & \mathrm{Bun}_h(X_{C})                                                    & \mathcal{E}(G).                                        
\end{tikzcd}
\]
\end{thm}

\begin{rem} Note that the fibers of $\mathrm{proj}$ consists of some projective bundle. The fact that the local period domain $\mathbb{P}^{h-1}$ fits into $\mathrm{PerDom}$ as fibers justifies our definition of $\pi_{\mathrm{GH}}^{\mathrm{glob}}$ as a global crystalline period map. \end{rem}

\section{Further Questions}
\begin{q} What is the analogue of the classical crystalline period map, replacing $H^1_{dR}$ with $QH^1_{dR}$? Is there a period map for $\omega \in \Delta_R$ with quotient the classical crystalline period map $\omega_{G} \hookrightarrow H^{dR}(G)$? Jonas McCandless suggested that $v$-sheaves might help. There is a difficulty because it is unclear how to recover the differentials that turn $\Omega^i$ into $H^i_{DR}(-)$ from the crystal $\Delta_R$.  \end{q}

\begin{rem} The inclusion $\omega_{G} \hookrightarrow H^1_{dR}(G^t)$, where $G^t$ is the Cartier dual, is equivalent to the inclusion we took of Hodge-structures $\fLie(VG) \to \fLie(G)$, in our definition of the period map. They differ only by the choice of definition of the Dieudonn\'{e} crystal. \end{rem}

\begin{q} Behold the Shimura varieties $S_K$ or their perfectoid bretheren $\mathrm{Sh}^\flat_\infty$ constructed by many (Kottwitz, Harris-Taylor, Carayol, Beilinson, Drinfel'd) as the global object in the global-local approach of computing the cohomology of Lubin-Tate space to realize the local and Jacquet-Langlands correspondence. How is $S_K$ related to the stack $\mathrm{Hecke}$ we use as the source of our global period map? Is there a relation between $H^*_{\acute{e}t}(S_K^{\text{ }\textrm{an}}; \overline{\Q}_\ell)$ and $H^*_{\acute{e}t}(\mathrm{Hecke}; \overline{\Q}_\ell)$? \end{q}

We discuss a partial answer to this question. Firstly, section 7.1 \cite{Fargues} describes the following ecosystem, where $g$ is the answer to the question above.

\[
\begin{tikzcd}
                                                                           & \mathrm{Flag}^{\diamond} \arrow[rd, "i"] &                                                \\
\mathrm{Sh}^{\flat}_\infty \arrow[rd, "f"'] \arrow[ru, "\pi_{HT}^{\diamond}"] \arrow[rr, "g"] &                               & \mathrm{Hecke} \arrow[ld, "\overleftarrow{h}"] \\
                                                                           & \mathrm{Bun}_G                &                                              
\end{tikzcd}
\]

\begin{q} What is the equivariance of the global crystalline period map? It will be na\"{i}vely equivariant on each fiber, but may have a twisted group action viewing the whole bundle at once.  \end{q}

\bibliographystyle{acm}
\bibliography{crysper_draft_the_final}

\end{document}